\documentclass{amsart}
\usepackage{babel}
\usepackage{amsmath, amssymb}
\usepackage{xcolor}
\usepackage{mathrsfs}
\usepackage{mathtools}
\usepackage[T1]{fontenc}
\usepackage{graphicx}
\usepackage{mathtools}
\usepackage{amssymb}
\usepackage{amsthm}
\usepackage{thmtools}
\usepackage{xcolor}
\usepackage{hyperref}
\usepackage{tikz,pgfplots}
\usepackage{tikz-cd}
\usetikzlibrary{calc,through,intersections,patterns,patterns.meta,arrows.meta,spy}

\newtheorem{theorem}{Theorem}[section]
\newtheorem{lemma}[theorem]{Lemma}
\newtheorem{corollary}[theorem]{Corollary}

\theoremstyle{definition}
\newtheorem{definition}[theorem]{Definition}
\newtheorem{example}[theorem]{Example}

\theoremstyle{remark}
\newtheorem{remark}[theorem]{Remark}

\numberwithin{equation}{section}

\newcommand{\diver}{{{\mathrm{div}}}}

\pgfplotsset{compat=1.8}

\begin{document}

\title[A Supersymmetric Perspective]{ Dirichlet--Neumann duality for the Basic Spectrum of Riemannian Submersions: A Supersymmetric Perspective}

\author{Vicent Gimeno i Garcia}
\address{Department of Mathematics, Universitat Jaume I-IMAC, E-12071, 
Castell\'{o}, Spain}
\email{gimenov@mat.uji.es}

\author{Paulo Henryque da Costa Silva}
\address{Department of Mathematics, Universidade Federal do Cear\'{a}, 
60440-900, Fortaleza, Brazil}
\email{paulohenryque@alu.ufc.br}



\date{\today}

%%%%%%%%%%%%%%%%%%%%%%%%%%%%%%%%%%%%%%%%%%%%%%%%%%%%%%%%%%%%%%%%%%%%%%%%%%%%%%%%%%%%%%%%%%%%%%%%%%%%%%%%%%%%%%%%%%%%%%%%%%%%%%%%%%%%%%%%%%%%%%%%
\begin{abstract}
This manuscript investigates the spectral geometry of Riemannian submersions whose fibers have a basic mean curvature. By restricting the Laplace--Beltrami operator to the space of basic functions, we reduce the spectral problem on $M$ to the spectral problem for a weighted Laplacian on the base manifold, where the weight is determined by the fiber-volume function $S$. We derive a summation formula for the reciprocal of the basic Dirichlet eigenvalues (Basel-type series). Furthermore, using the framework of Supersymmetric Quantum Mechanics (SUSYQM), we establish a supersym\-me\-tric duality relating the basic Dirichlet and Neumann spectra under the trans\-for\-ma\-tion $S \mapsto 1/S$. 
\end{abstract}
%%%%%%%%%%%%%%%%%%%%%%%%%%%%%%%%%%%%%%%%%%%%%%%%%%%%%%%%%%%%%%%%%%%%%%%%%%%%%%%%%%%%%%%%%%%%%%%%%%%%%%%%%%%%%%%%%%%%%%%%%%%%%%%%%%%%%%%%%%%%%%%%

\maketitle

%%%%%%%%%%%%%%%%%%%%%%%%%%%%%%%%%%%%%%%%%%%%%%%%%%%%%%%%%%%%%%%%%%%%%%%%%%%%%%%%%%%%%%%%%%%%%%%%%%%%%%%%%%%%%%%%%%%%%%%%%%%%%%%%%%%%%%%%%%%%%%%%
\section{Introduction}\label{sec:intro}
%%%%%%%%%%%%%%%%%%%%%%%%%%%%%%%%%%%%%%%%%%%%%%%%%%%%%%%%%%%%%%%%%%%%%%%%%%%%%%%%%%%%%%%%%%%%%%%%%%%%%%%%%%%%%%%%%%%%%%%%%%%%%%%%%%%%%%%%%%%%%%%%

The interplay between the geometric properties of a manifold and the spectrum of its Laplacian operator $\triangle={\rm div}\circ \nabla$ is a central theme in spectral geometry \cite{Chavel, Kac, Pol2, YauReview}. In this work, we apply a theoretical framework inspired by Supersymmetric Quantum Mechanics (SUSYQM), as developed in the seminal works of Witten \cite{Witten1982} and Cooper et al. \cite{CooperKhareSukhatme1995}, to study the basic spectrum of a complete Riemannian manifold $(M, g)$. More precisely, we consider manifolds admitting a Riemannian submersion 
\begin{equation}
\pi\colon M \to B,
\end{equation}
where the fibers possess a basic mean curvature vector $\vec{H}$, meaning that the mean curvature vector field is the horizontal lift of a vector field on the base manifold.

The primary focus of our research is the set of eigenvalues associated with basic eigenfunctions, denoted by $\sigma_{\rm basic}$. These eigenfunctions are defined as the pullbacks of functions $\widetilde{f}: B \to \mathbb{R}$ from the base. The restriction of the Laplacian to this subspace, called the basic Laplacian $\triangle_{\rm basic}$, allows a significant reduction in the complexity of the spectral problem. When the fibers are compact without boundary, their finite volume can be denoted by $S(p)=\operatorname{vol}(\pi^{-1}(p))$, and the problem becomes equivalent to studying a weighted Laplacian on the base $B$ with weight $S$.

Using a weighted decomposition inspired by the SUSYQM technique, we prove that the following second-order differential operator 
$$
\mathscr{L}_S^{-}:\mathfrak{X}(B)\to \mathfrak{X}(B),\quad X\mapsto\mathscr{L}_S^{-}X=-S \nabla\left( \frac{\diver(X)}{S} \right),
$$ 
defined on vector fields $\mathfrak{X}(B)$ of the base manifold $B$, together with the constrained boundary problem
\begin{equation}\label{constrained}
\left\{
\begin{aligned}
\mathscr{L}_S^{-}X &= \lambda X \quad &&\text{in } D\subset B,\\
\diver(X) &= 0 \quad &&\text{on } \partial D,
\end{aligned}
\right.
\end{equation}
admits a finite-dimensional space of eigenfields for $\lambda>0$ and for $D$ a precompact domain of $B$ with smooth boundary $\partial D$. (This constrained eigenvalue problem is well-posed within the subspace of fields satisfying the divergence-free boundary condition.) Furthermore, obtaining the spectrum of the weighted Laplacian, and hence the basic spectrum, with Dirichlet boundary conditions is equivalent to finding the eigenvalues for \eqref{constrained}. Indeed, there is a map from the eigenvectors of the weighted Laplacian problem with Dirichlet boundary to eigenfields of \eqref{constrained}.

\medskip
\noindent\textbf{Main results of this paper.} 
The strongest spectral consequences of this type of supersymmetric factorization arise in the one-dimensional reduction because $\mathfrak{X}(B)\simeq C^\infty(B)$. Hence, the weighted basic Laplacian becomes a Sturm--Liouville operator, and the duality of boundary conditions becomes a Dirichlet--Neumann duality. In particular, when the base manifold is $1$-dimensional, this leads to a formal duality between Dirichlet and Neumann spectra under the transformation $S(t) \mapsto 1/S(t)$ (see Figure \ref{Fig:DNintro}), and we can state the following theorem (which is one of the main novelties of this work).

\begin{figure}[htb!]
\centering
	\begin{tikzpicture}[scale=0.58,>=stealth]
		
		\begin{scope}[scale=1.7,x=1.6cm,yshift=1.4cm]
			\def\u{1}
			\def\v{1}
			
			\draw[domain=-3.5:3.5,line width=1pt,samples=500] plot ({\x},{exp(-pow(\x,2))});
			\draw[domain=-3.5:3.5,line width=1pt,samples=500] plot ({\x},{exp(-pow(\x,2))});
			\draw[domain=-3.5:3.5,line width=1pt,samples=500] plot ({\x},{-exp(-pow(\x,2))});
			\draw[domain=-3.5:3.5,line width=1pt,samples=500] plot ({\x},{-exp(-pow(\x,2))});

			\draw[blue,line width=1pt] (0,1) arc (90:270:0.3cm and \v cm);
			\draw[blue,dashed] (0,1) arc (90:-90:0.3cm and \v cm);
			
			\def\k{1}
			\draw[blue, line width=1pt,samples=500] (\k,{exp(-pow(\k,2))}) arc (90:270:0.1cm and 0.36*\v cm);
			\draw[blue, dashed,samples=500] (\k,{exp(-pow(\k,2))}) arc (90:-90:0.1cm and 0.36*\v cm);
			
			\draw[blue, line width=1pt,samples=500] (-\k,{exp(-pow(\k,2))}) arc (90:270:0.1cm and 0.36*\v cm);
			\draw[blue, dashed] (-\k,{exp(-pow(\k,2))}) arc (90:-90:0.1cm and 0.36*\v cm);
			
			\def\t{0.7}
			\draw[blue, line width=1pt,samples=500] (\t,{exp(-pow(\t,2))}) arc (0:-180:1.1*\v cm and 0.15cm);
			\draw[blue, dashed,samples=500] (\t,{exp(-pow(\t,2))}) arc (0:180:1.1*\v cm and 0.15cm);
			
			\draw[blue, line width=1pt,samples=500] (\t,{-exp(-pow(\t,2))}) arc (0:-180:1.1*\v cm and 0.15cm);
			\draw[blue, dashed,samples=500] (\t,{-exp(-pow(\t,2))}) arc (0:180:1.1*\v cm and 0.15cm);
		\end{scope}
		     
			\draw[->] (0,-0.1) -- (0,-2) node[right, pos=0.5] {$\pi_{1}$};
			\draw (-9.8,-2.5) -- (9.8,-2.5);
			
			\draw[line width=1pt] (0,-2.4) -- (0,-2.6) node[below] {$0$};
			
			\draw[<-] (0,-3.4) -- (0,-5.2)   node[right, pos=0.5] {$\pi_{2}$} ;
			
			\node (aux1) at (6,0.2) {$S_{M_1}(t) = \frac{\displaystyle 1}{ \displaystyle S_{M_2}(t)}  $};

            \node (aux1) at (-9.2,3.0) {$M_1$};

            \node (aux1) at (-9.2,-3.7) {$M_2$};

		\begin{scope}[yshift=-5.5cm]
			\def\q{9.7}
			\draw[dashed] (\q,1) arc (90:270:0.2cm and 1.5 cm);
			\draw[line width=1pt] (\q,1) arc (90:-90:0.2cm and 1.5 cm);
			
			\draw[line width=1pt] (-\q,1) arc (90:270:0.2cm and 1.5 cm);
			\draw[dashed] (-\q,1) arc (90:-90:0.2cm and 1.5 cm);
			
			\draw[line width=1pt] (0,-0.3) to[out=0,in=200] (\q,1);
			\draw[line width=1pt] (0,-0.3) to[out=180,in=-20] (-\q,1);
			
			\draw[line width=1pt] (0,-0.6) to[out=0,in=160] (\q,-1.98);
			\draw[line width=1pt] (0,-0.6) to[out=180,in=20] (-\q,-1.98);
		\end{scope}	
	\end{tikzpicture}
    \caption{\small{Schematic representation of the correspondence between nonzero Dirichlet and Neumann eigenvalues induced by the transformation $S(t) \mapsto 1/S(t)$}.}
    \label{Fig:DNintro}
\end{figure}
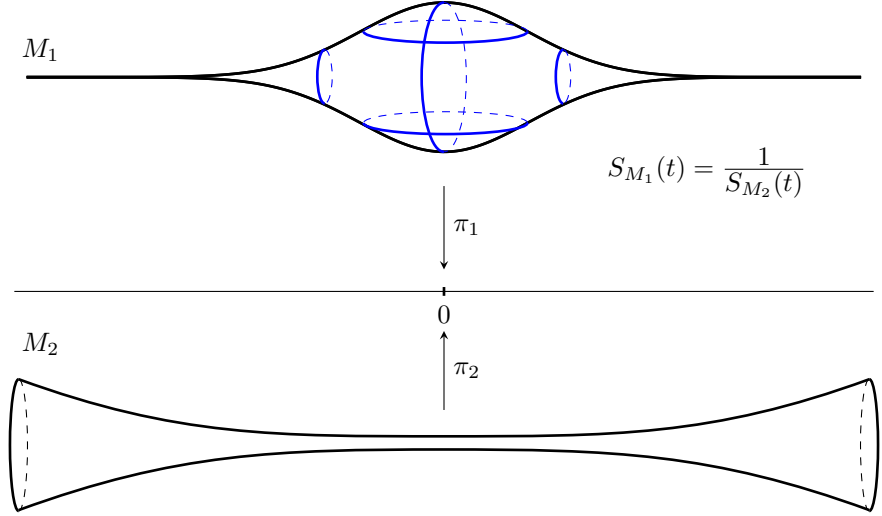

\begin{theorem}\label{DNcorres}
Let $\pi_1: M_1\to \mathbb{R}$ and $\pi_2: M_2\to \mathbb{R}$ be two Riemannian submersions with compact fibers of basic mean curvature.  Let $\Omega_1:=\pi^{-1}_1([a,b])$, $\Omega_2:=\pi^{-1}_2([a,b])$  and suppose that the fiber volumes satisfy 
$$
S_1(t)=\frac{1}{S_2(t)},\quad \forall t\in[a,b],
$$
where $
S_1(t)={\rm vol}(\pi_1^{-1}(t))
$ and $
S_2(t)={\rm vol}(\pi_2^{-1}(t))
$. 
Then
$$
\sigma_{\rm basic}^{\mathcal{D}}(\Omega_1)=\sigma_{\rm basic}^\mathcal{N}(\Omega_2)-\{0\},\quad \sigma_{\rm basic}^{\mathcal{D}}(\Omega_2)=\sigma_{\rm basic}^\mathcal{N}(\Omega_1)-\{0\},
$$
where $\sigma_{\rm basic}^{\mathcal{D}}(\Omega)$ denotes the Dirichlet basic spectrum of $\Omega$ and $\sigma_{\rm basic}^\mathcal{N}(\Omega)$ denotes the Neumann basic spectrum of $\Omega$.
\end{theorem}

Spectral comparisons between the Dirichlet and Neumann problems have been extensively studied; see, for instance, \cite{LevineWeinberger1986, Friedlander1991, Hansson2008, FrankLaptev2010, ArendtMazzeo2012, Rohleder2021Schrodinger, Rohleder2025, HuaMunchZhang2026}. In the spirit of Payne's inequalities \cite{Payne1955} relating Dirichlet and Neumann eigenvalues, Theorem \ref{DNcorres} yields the following result

\begin{corollary}\label{PAYNE}
Let $\pi: M\to \mathbb{R}$ be a Riemannian submersion with compact fibers of basic mean curvature. Let $\Omega=\pi^{-1}([a,b])$ with $\partial \Omega=\pi^{-1}(a)\cup\pi^{-1}(b)$. The following assertions hold
\begin{enumerate}
    \item If $S$ is log-convex, i.e, $(\log{S(t)})''\geq0$ then  $\lambda_{2,{\text{basic}}}^{\mathcal{N}}(\Omega) \leq \lambda_{1, \text{basic}}^{\mathcal{D}}(\Omega)$. This implies $\lambda_{2}^{\mathcal{N}}(\Omega) \leq \lambda_{1}^{\mathcal{D}}(\Omega)$.
    \item If $S$ is log-concave, i.e, $(\log{S(t)})'' \leq 0$ then  $\lambda_{1,{\text{basic}}}^{\mathcal{D}}(\Omega) \leq \lambda_{2, \text{basic}}^{\mathcal{N}}(\Omega)$.
\end{enumerate}
Moreover, the equality $\lambda_{2,{\text{basic}}}^{\mathcal{N}}(\Omega) = \lambda_{1, \text{basic}}^{\mathcal{D}}(\Omega)$ 
in the above items holds if and only if $(\log{S(t)})''=0$. 
\end{corollary}

\begin{remark} The exponential fiber-growth behavior illustrated in Figure \ref{Fig:DNintro} is, in a certain sense, optimal for our framework. Indeed, under the usual supersymmetric correspondence between the basic Laplacian $\triangle_{\rm basic}$ and the associated Schr\"{o}dinger operator, discussed in detail in Section \ref{sec:susy2}, exponential fiber growth corresponds precisely to the case where the effective Schr\"{o}dinger potential is constant.
\end{remark}

Furthermore, this paper addresses the eigenvalue problem on bounded domains $\Omega = \pi^{-1}([a,b])$ with Dirichlet or Neumann boundary conditions. We derive exact formulas for the sum of the reciprocals of the eigenvalues:

\begin{theorem}\label{teo:sum}
Let $\pi: M\to \mathbb{R}$ be a Riemannian submersion with compact fibers of basic mean curvature. Let $\Omega=\pi^{-1}([a,b])$ with $\partial \Omega=\pi^{-1}(a)\cup\pi^{-1}(b)$, 
and let 
 $$
 \sigma_{\rm basic}^\mathcal{D}(\Omega)=\left\{\lambda^\mathcal{D}_1(\Omega)>0, \lambda^\mathcal{D}_2(\Omega), \cdots \right\}
 $$
 be the basic spectrum of $\Omega$ with the Dirichlet boundary condition. Then:
\begin{equation*}
\sum_{k=1}^{\infty}\frac{1}{\lambda_k^{\mathcal{D}}(\Omega)}= \left( {\int_{a}^b\frac{dt}{S(t)}} \right)^{-1}\int_a^b S(x)\left(\int_a^x\frac{dt}{S(t)}\int_x^b\frac{dt}{S(t)}\right)dx < \infty.
\end{equation*}
\end{theorem}

This theorem provides a geometric generalization of the classical Basel problem. For instance, if $S(t)=2\pi$ on $[0,L]$, we (re)deduce $\sum 1/k^2 = \pi^2/6$. For the case $S(t)=2\pi e^{-\alpha t}$, in this case, for $z = (\alpha L)/(2\pi)$ we obtain the classical Mittag-Leffler expansion for $\coth(z)$:
\begin{equation*}
\sum_{k=1}^{\infty} \frac{1}{z^2+k^2} = \frac{\pi}{2z} \coth{(\pi z)} -
\frac{1}{2z^2} \cdot
\end{equation*}

Finally, the aforementioned transformation $S(t)\mapsto 1/S(t) $ allows us, by means of Theorem \ref{teo:sum}, to compute the sum of the reciprocals of the Neumann eigenvalues, as stated in the following corollary:

\begin{corollary}\label{cor:sumneumann}
 Let $\pi: M\to \mathbb{R}$ be a Riemannian submersion with compact fibers of basic mean curvature. Let $\Omega=\pi^{-1}([a,b])$  and let 
 $$
 \sigma_{\rm basic}^\mathcal{N}(\Omega)=\left\{\lambda^\mathcal{N}_0(\Omega)=0,\lambda^\mathcal{N}_1(\Omega)>0, \lambda^\mathcal{N}_2(\Omega),\cdots \right\}
 $$
 be the basic spectrum of $\Omega$ with Neumann boundary conditions. Then
 $$
 \sum_{k=1}^{\infty}\frac{1}{\lambda_k^{\mathcal{N}}(\Omega)}= \left( {\int_{a}^bS(t)dt} \right)^{-1}\int_a^b \frac{1}{S(x)}\left(\int_a^xS(t)dt\int_x^bS(t)dt\right)dx<\infty.
 $$
\end{corollary}

To the best of our knowledge, no previous work has established this duality nor the exact summation formulas in the context of Riemannian submersions with basic mean curvature. The closest results are the spectral estimates for submersions by Polymerakis \cite{Polymerakis2021} and the capacity-based trace formulas for one-dimensional weighted Laplacians (see \cite{bessa2016, GriExp}). Our contribution is to unify these ideas within a SUSYQM-inspired algebraic framework and to extract explicit spectral identities.

\subsection*{Outline of the paper:}

The structure of this paper is organized as follows. In Section \ref{sec:sub}, we review the fundamental concepts of Riemannian submersions and the definition of the basic Laplacian. Section \ref{sec:susy2} introduces the Supersymmetric Quantum Mechanics formalism in order to obtain a supersymmetric formulation of the inverse weighted decomposition. In particular, the supersymmetric duality between Dirichlet and Neumann boundary conditions is presented. Section \ref{sec:sum} is dedicated to the derivation of the formula for the sum of reciprocal eigenvalues with Dirichlet boundary conditions. Finally, in Section \ref{Seccomparison}, we establish the comparison theorem for manifolds with constant volume growth and compute the reciprocal sum of the eigenvalues for precompact domains on the catenoid and the pseudosphere.

%%%%%%%%%%%%%%%%%%%%%%%%%%%%%%%%%%%%%%%%%%%%%%%%%%%%%%%%%%%%%%%%%%%%%%%%%%%%%%%%%%%%%%%%%%%%%%%%%%%%%%%%%%%%%%%%%%%%%%%%%%%%%%%%%%%%%%%%%%%%%%%%
\section{Riemannian submersions, Laplacian Operator, Basic Functions and Submersions with Fibers of Basic Mean Curvature}\label{sec:sub}
%%%%%%%%%%%%%%%%%%%%%%%%%%%%%%%%%%%%%%%%%%%%%%%%%%%%%%%%%%%%%%%%%%%%%%%%%%%%%%%%%%%%%%%%%%%%%%%%%%%%%%%%%%%%%%%%%%%%%%%%%%%%%%%%%%%%%%%%%%%%%%%%

In this section, we recall the geometric foundations of Riemannian submersions, with particular emphasis on those having a basic mean curvature vector field, a condition that simplifies the spectral analysis. We begin by recalling the standard definition of a Riemannian submersion:

\begin{definition}
A surjective map $\pi: (M, g_M) \to (B, g_B)$ between two Riemannian manifolds is called a \emph{Riemannian submersion} if, for every point $q \in M$, the differential $d\pi_q: T_qM \to T_{\pi(q)}B$ preserves the length of vectors orthogonal to the kernel of $d\pi_q$. 
\end{definition}

More specifically, for each $q \in M$, the tangent space can be decomposed into the direct sum $T_qM = \mathcal{V}_q \oplus \mathcal{H}_q$, where:
\begin{itemize}
 \item $\mathcal{V}_q = \ker(d\pi_q)$ is the \emph{vertical space}, tangent to the fiber $F_p = \pi^{-1}(p)$ (where $p=\pi(q)$).
 \item $\mathcal{H}_q = \mathcal{V}_q^\perp$ is the \emph{horizontal space}.
\end{itemize}
The map $\pi$ is a Riemannian submersion if the restriction $d\pi_q|_{\mathcal{H}_q}: \mathcal{H}_q \to T_{\pi(q)}B$ is an isometric isomorphism.

A function $f \in C^\infty(M)$ is said to be \emph{basic} if it is constant along the fibers of the submersion. This is equivalent to the existence of a function $\widetilde{f} \in C^\infty(B)$ such that $f = \widetilde{f} \circ \pi$. The Laplacian of $M$ acting on such functions has a particularly interesting structure.

A submersion is said to have \emph{fibers with a basic mean curvature vector} if the vector $d\pi_q(H_q)$ depends only on the base $B$; that is, there exists a vector field $\widetilde{H}$ on $B$ such that $\widetilde{H}_{\pi(q)} = d\pi_q(H_q)$ for all $q \in M$.

In this scenario, if the fibers are compact and we denote their volume by $S(p) = \text{vol}(\pi^{-1}(p))$, it has been proven (see \cite{Bordoni2006}) that the field $\widetilde{H}$ is intimately linked to the volume of the fibers:
\begin{equation}
\widetilde{X}(S(p)) = -\langle \widetilde{H}_p, \widetilde{X}_p \rangle_B S(p), \quad \forall \widetilde{X} \in \mathfrak{X}(B).
\end{equation}
This implies $\widetilde{H} = -\nabla^B (\log S)$.

There exists a wide variety of Riemannian submersions whose fibers possess a basic mean curvature vector field. This geometric condition is not restrictive to a single family of manifolds; rather, it appears naturally in several fundamental constructions. We provide concrete examples focusing on two main frameworks: warped products, where the geometry of the fiber is scaled by a function on the base manifold, and specific constructions on Lie groups, which leverage algebraic symmetries to satisfy the basicity of the mean curvature.

\begin{example}\label{expwap}\emph{Examples of Riemannian submersions with fibers having basic mean curvature vector field}
\begin{enumerate}
 \item \emph{Warped Products as Riemannian Submersions}. Warped products represent one of the most natural sources of Riemannian submersions with basic mean curvature. In these models, the metric on a product manifold is adapted such that the fiber volume  varies consistently according to a warping function defined on the base manifold.

Let $(B, g_B)$ and $(F, g_F)$ be Riemannian manifolds, and let $w: B \to (0, \infty)$ be a smooth \emph{warping function}. The warped product $M = B \times_w F$ is the product manifold equipped with the metric:
\[ g = \pi^* g_B + (w \circ \pi)^2 \sigma^* g_F \]
where $\pi: M \to B$ and $\sigma: M \to F$ are the natural projections. The map $\pi$ is a Riemannian submersion because its differential $d\pi$ is an isometry when restricted to the horizontal distribution $\mathcal{H} = (\ker d\pi)^\perp$.

The fibers $F_p = \pi^{-1}(p)$ are submanifolds of $M$. For any vertical vector fields $V, W \in \mathcal{V}$ (tangent to the fibers), the Levi--Civita connection $\nabla$ of $M$ satisfies (see \cite{Oneill} for instance):
\[ \nabla_V W = \nabla^F_V W - \frac{g(V, W)}{\omega} \nabla\omega, \]
where $\omega=w\circ \pi$. The second fundamental form $\mathrm{II}$ of the fibers is the horizontal component of this connection:
\[ \mathrm{II}(V, W) = (\nabla_V W)^\perp = - \frac{g(V, W)}{\omega} \nabla\omega. \]
This identity shows that the fibers are totally umbilical submanifolds, as $\mathrm{II}$ is proportional to the metric $g$. The mean curvature vector $\vec{H}$ is the trace of $\mathrm{II}$. Let $\{E_1, \dots, E_n\}$ be a local orthonormal basis for the fiber:
\[ \vec{H} = \sum_{i=1}^n \mathrm{II}(E_i, E_i) = -n\sum_{i=1}^n \frac{1}{\omega}\nabla\omega,\quad n = \dim F. \]
Since 
$$
d\pi\left(-n \frac{\nabla\omega}{\omega}\right)=-n\frac{\nabla^Bw}{w},
$$
we conclude that any warping product admits the structure of a submersion with fibers of basic mean curvature vector field. Observe that when $F$ has finite volume,
$$
S(p)= {\rm vol}(\pi^{-1}(p))=w(p)^n{\rm vol}(F),
$$
and as stated,
$$
d\pi(\vec H)=-\frac{\nabla^B S}{S}.
$$

In the specific case of a 1-dimensional base manifold $\mathbb{R}\times_wF$, the mean curvature of the fibers is given by
$$
\vec{H}=-n\frac{w'(t)}{w(t)}=-\frac{S'(t)}{S(t)}\partial t.
$$
\item \emph{Lie Group constructions}. Beyond the warped product structure, it is known (see \cite{Polymerakis2021} for instance) that whenever a Lie group $G$ acts smoothly, freely, and properly via isometries on a Riemannian manifold $M$ and ${\rm dim}(G)<{\rm dim}(M)$, the projection $\pi:M\to B:=M/G$ is a Riemannian submersion with fibers of basic mean curvature vector field.

We can moreover construct Riemannian submersions with fibers of basic mean curvature vector field by defining specific metrics on the product of the real line and a Lie group $M=\mathbb{R}\times G$. By choosing left-invariant vector fields and allowing the metric coefficients to depend on the parameter $t$, we obtain a class of submersions where the basicity of the mean curvature of the fibers is guaranteed by the algebraic structure of the group.

Let $G$ be a compact $n$-dimensional Lie group with Lie algebra $\mathfrak{g}$. Let $\{X_1,\cdots, X_n\}$ be a family of non-vanishing left invariant vector fields of $G$ such that
$$
\mathfrak{g}={\rm span}\left\{X_1,\cdots,X_n\right\}.
$$
Consider the manifold $M=\mathbb{R}\times G$ endowed with the following metric $\langle,\rangle_M$:
$$\begin{array}{lclr}
\langle \partial t,\partial t\rangle&=&1& \\
\langle \partial t,X_i\rangle&=&0& \forall i\\
\langle X_i,X_j\rangle&=&0 & \text{ for } i\neq j\\
\langle X_i,X_i\rangle&=&h^2_i(t)q^2_i(g)& \forall i
\end{array}
$$
where $h_i\in C^{\infty}(\mathbb{R},\mathbb{R}_+)$ and $q_i\in C^{\infty}(G,\mathbb{R}_+)$. The map $\pi:M\to \mathbb{R}$ is a Riemannian submersion with compact fibers. To verify that the mean curvature of the fibers is a basic vector field, we use the Koszul formula:
 $$
 2\langle \nabla_{X_i}X_i,\partial_t\rangle=-\partial_t\left(\langle X_i,X_i\rangle\right)=-2h_i(t)\dot h_i(t) q_i^2(g).
 $$
Hence, the mean curvature vector field $H$ is:
 $$
 \begin{aligned}
 H=&{\rm tr}(\mathrm{II})=\sum_{i=1}^n\mathrm{II}\left(\frac{X_i}{\Vert X_i\Vert},\frac{X_i}{\Vert X_i\Vert}\right)=\sum_{i=1}^n\frac{1}{\Vert X_i\Vert^2}\mathrm{II}(X_i,X_i)\\
 =&\sum_{i=1}^n\frac{1}{\Vert X_i\Vert^2}\langle \nabla_{X_i}X_i,\partial t\rangle \partial_t=-\left(\sum_{i=1}^n\frac{\dot h_i(t)}{h_i(t)}\right)\partial_t .
 \end{aligned}
$$
This leads to the relationship:
 $$
d\pi(H)=-\frac{\nabla S}{S}=-\frac{\dot S}{S}
$$
by taking $S(t)=\prod_{i=1}^nh_i(t)$, confirming that these constructions satisfy the required conditions for our spectral analysis.
\end{enumerate} 
\end{example}

As shown in the literature (see \cite{Bessa2012}), the Laplacian of a basic function on a manifold admitting a Riemannian submersion can be expressed as:
\begin{equation}\label{baseinduced}
 (\triangle^M f)_q = (\triangle^B \widetilde{f})_p - \langle \nabla^B \widetilde{f}_p, d\pi_q(H_q) \rangle_B,
\end{equation}
where $H_q$ is the mean curvature vector of the fiber passing through $q$. Assuming that the submersion has fibers of basic mean curvature, we can transform the Laplacian of a basic function into the \emph{weighted Laplacian} (\cite{GriBook} for instance) of a function on the base manifold:
\begin{equation}\label{weightedlaplacian_en}
 (\triangle^M f)_q = \frac{1}{S(p)} \text{div} \left( S \cdot \nabla^B \widetilde{f} \right)_p = \triangle_{\mu}^B \widetilde{f},
\end{equation}
where $\mu$ is the measure $d\mu = S d\sigma_B$. This equivalence is fundamental, as it reduces the study of the spectrum of basic functions on $M$ (a potentially high-dimensional manifold) to a spectral problem on the base $B$ with a weight determined by the geometry of the fibers.

We focus on basic eigenfunctions of the Laplacian because, following the ideas of \cite[Theorem 1.6]{Bordoni2006}, the Sobolev space $H^{1}(M)$ is the direct sum of
$\mathcal{E}_c$ and $\mathcal{E}_0$ where
\begin{equation}\label{eq:EcE0}
\begin{aligned}
\mathcal{E}_c &:= \left\{ f \in H^1(M) \;:\; f=u\circ\pi \ \text{for some } u:B \to \mathbb{R}
\right\}; \\
\mathcal{E}_0 &:= \left\{ h \in H^1(M) \;:\; \int_{\pi^{-1}(x)} h(y)dv_{g_x}(y)=0 \right\}.
\end{aligned}
\end{equation}
This orthogonal decomposition $L^2(M)=\mathcal E_c\oplus \mathcal E_0$ is a reducing decomposition for the self-adjoint Laplace--Beltrami operator $\triangle_M$. Hence $\triangle_M=\triangle_c\oplus\triangle_0$, where
\[
\triangle_c = \left.\triangle_M\right|_{\mathcal D(\triangle_M)\cap\mathcal E_c}, \qquad \triangle_0 = \left. \triangle_M \right|_{\mathcal D(\triangle_M)\cap\mathcal E_0}.
\]
In particular,
\[
\sigma(\triangle_M)=\sigma(\triangle_c)\cup\sigma(\triangle_0).
\]

To analyze the spectral properties of $\triangle_{\mu}^B$, standard geometric estimates can be rigid. In the next section, we introduce an algebraic factorization framework inspired by Supersymmetric Quantum Mechanics to overcome this limitation. This technique allows us to factorize the weighted Laplacian into first-order operators and establish a powerful duality between different boundary value problems.

%%%%%%%%%%%%%%%%%%%%%%%%%%%%%%%%%%%%%%%%%%%%%%%%%%%%%%%%%%%%%%%%%%%%%%%%%%%%%%%%%%%%%%%%%%%%%%%%%%%%%%%%%%%%%%%%%%%%%%%%%%%%%%%%%%%%%%%%%%%%%%%%
\section{Inverse weighted decomposition on manifolds, proof of Theorem \ref{DNcorres} and Corollary \ref{cor:sumneumann}} \label{sec:susy2}
%%%%%%%%%%%%%%%%%%%%%%%%%%%%%%%%%%%%%%%%%%%%%%%%%%%%%%%%%%%%%%%%%%%%%%%%%%%%%%%%%%%%%%%%%%%%%%%%%%%%%%%%%%%%%%%%%%%%%%%%%%%%%%%%%%%%%%%%%%%%%%%%

This section introduces the technical details of the supersymmetric decomposition of the weighted Laplacian on the base manifold. The base manifold is endowed with a density $d\mu=S(x)dv_g(x)$ where $dv_g$ is the Riemannian volume element of $(B,g_B)$. The weighted Laplacian operator acts on smooth functions by 
\begin{equation}
 \mathscr{L}^+_S \varphi \coloneqq -\frac{1}{S} \diver \left( S \nabla \varphi \right)=-\triangle_\mu \varphi.
\end{equation}
The notation $\mathscr{L}_S^+$ will become clearer shortly. We provide an inverse weighted decomposition of Sturm--Liouville type (see \cite{LevitanSargsjan1991, Zettl2005}) for the operator $\mathscr{L}^+_S$, based on the supersymmetric factorization method of \cite{InfeldHull1951, CooperKhareSukhatme1995}. 

First, we introduce the following two first-order operators $\mathcal{A}^\dag_S$ and $\mathcal{A}_S$ on the space of smooth functions $C^\infty(B)$ and smooth vector fields $\mathfrak{X}(B)$: $\mathcal{A}^\dag_S\colon \mathfrak{X}(B) \to C^{\infty}(B)$ and $\mathcal{A}_S\colon C^{\infty}(B) \to \mathfrak{X}(B)$ given by
\begin{equation*}
\begin{array}{ll}
  X\mapsto \mathcal{A}^\dag_S X \coloneq \frac{1}{S}\diver(X),&\varphi\mapsto\mathcal{A}_S \varphi \coloneq -S\nabla\varphi.
\end{array} 
\end{equation*}
The Laplacian operator is the composition of these two operators, namely,
$$
\mathscr{L}_S^+=\mathcal{A}_S^\dag\circ \mathcal{A}_S.
$$

Our supersymmetry-inspired approach is to study the properties of $\mathscr{L}_S^+$ using the properties of $\mathcal{A}^\dag_S$ and $\mathcal{A}_S$. Note that we need to consider not only the space of smooth functions, but also the space of smooth vector fields. This is natural from the perspective of SUSY, where a bosonic space (functions) is accompanied by a fermionic space (differential forms). In the present setting, the musical isomorphism $g_B^\flat: \mathfrak{X}(B) \to \Omega^1(B)$ identifies vector fields with 1-forms, so the pair $(C^\infty(B), \mathfrak{X}(B))$ mimics the structure of a supersymmetric partner system. In this paper, we do not dwell on this physical interpretation beyond the algebraic factorization and instead treat them simply as the space of functions and the space of vector fields.

\begin{remark}[Schrödinger type operator]
Once the decomposition $\mathscr{L}_S^+=\mathcal{A}_S^\dag\circ \mathcal{A}_S$ is established, the standard supersymmetric framework can be applied. The reader may ask for a Schr\"{o}dinger-type operator and decomposition to look for superpotentials and supercharges. This part of the theory is not required for what is used in this paper. Nevertheless, we remark here that by using the unitary map
\[
T\colon L^2(B,S(x)dv_{g})\longrightarrow L^2(B,dv_{g}),
\qquad Tf=S^{1/2}f,
\]
the operator $\mathscr{H}_S: =T\mathscr L_S^+T^{-1}$ is a Schr\"{o}dinger-type operator. Indeed, given $\varphi\in C_0^\infty(B)$, let $f=T^{-1}\varphi=S^{-1/2}\varphi=e^{-\psi/2}\varphi$; then
\[
T\mathscr L^+_ST^{-1}=e^{\psi/2}(\mathscr L_S^{+}(e^{-\psi/2}\varphi)) =e^{\psi/2}\left( -\triangle_B(e^{-\psi/2}\varphi)
-\left\langle \nabla\psi,\nabla(e^{-\psi/2}\varphi) \right\rangle\right),
\]
with $\psi \coloneq \log S$. A direct computation gives
\begin{align}
\nabla\bigl(e^{-\psi/2}\varphi\bigr)
&=
e^{-\psi/2}
\left(
\nabla\varphi-\frac12\varphi\nabla\psi
\right), \label{eq57}\\
\triangle_B\bigl(e^{-\psi/2}\varphi\bigr)
&=
e^{-\psi/2}
\left[ \triangle_B\varphi - \langle\nabla\psi,\nabla\varphi\rangle
+ \left( \frac14|\nabla\psi|^2 - \frac{1}{2}\triangle_B \psi \right)\varphi
\right]. \label{eq58}
\end{align}
Taking the inner product of \eqref{eq57} with $\nabla\psi$, we obtain
\[
\left\langle \nabla\psi,\nabla\bigl(e^{-\psi/2}\varphi\bigr) \right\rangle
= e^{-\psi/2} \left[ \langle\nabla\psi,\nabla\varphi\rangle -
\frac{1}{2}|\nabla\psi|^2\varphi
\right].
\]
Multiplying both sides by \(e^{\psi/2}=S^{1/2}\), we conclude that
\[
T\mathscr L_S^+T^{-1}\varphi = -\triangle_B\varphi + \left(
\frac{1}{4}|\nabla\psi|^2 +\frac{1}{2}\triangle\psi
\right)\varphi.
\]
Since $\psi=\log S$, this gives
\begin{equation}\label{SDcorres}
\mathscr{H}_S=T\mathscr L_ST^{-1} = - \triangle_{B} +\underbrace{\frac{1}{4}|\nabla\log S|^2
+\frac{1}{2}\triangle\log S}_{ \coloneqq  \mathcal{V}_S}.
    \end{equation}
Thus $\mathscr L^+_S$, and therefore the positive Laplace--Beltrami
operator $\triangle_M$ restricted to basic functions, is unitarily
equivalent to the Schr\"{o}dinger operator \(\mathscr H_S\). 

The Schr\"{o}dinger operator $\mathscr{H}_S$ with geometric potential depending only on the volume of the fibers has been studied in \cite{Polymerakis2021}. There, it is proved that the basic spectrum of the Laplacian on a Riemannian submersion with fibers having a basic mean curvature vector field is discrete if and only if the spectrum of $\mathscr{H}_S$ is discrete as well.
\end{remark}

As usual, we will work in the space of smooth functions with compact support, $C^\infty_0(B)$  or smooth vector fields with compact support $\mathfrak{X}_0(B)$.
The next step is to endow the space $V=C^\infty_0(B)\times \mathfrak{X}_0(B)$ with an inner product
$
\langle\hskip-2pt\langle \cdot , \cdot \rangle\hskip-2pt\rangle : V \times V \to \mathbb{R}$ given by
$$
\left\langle\mkern-4mu\left\langle (\varphi_1,X_1), (\varphi_2,X_2) \right\rangle\mkern-4mu\right\rangle := \int_B \varphi_1\varphi_2 \, d\mu + \int_B \langle X_1, X_2 \rangle \, d\mu',
$$
with $d\mu=Sdv_g$ and $d\mu'=\frac{1}{S}dv_g$.
In $V$ we can define two operators 
$$
\begin{aligned}
&\mathcal{Q}^\dag:V\to V, \quad (\varphi,X)\mapsto \mathcal{Q}^\dag (\varphi,X)\coloneq(\mathcal{A}^\dag_S(X),0),\\
&\mathcal{Q}:V\to V, \quad (\varphi,X)\mapsto \mathcal{Q} (\varphi,X)\coloneq(0, \mathcal{A}_S(\varphi)).
\end{aligned}
$$
The operators $\mathcal{Q}^\dag$ and $\mathcal{Q}$ are duals (adjoints) in the following sense
$$
\begin{aligned}
    \left\langle\mkern-4mu\left\langle \mathcal{Q}^\dag(\varphi_1,X_1), (\varphi_2,X_2) \right\rangle\mkern-4mu\right\rangle=&\left\langle\mkern-4mu\left\langle (\mathcal{A}^\dag(X_1),0), (\varphi_2,X_2) \right\rangle\mkern-4mu\right\rangle=\int_B\frac{1}{S}\diver(X_1)\varphi_2\,  d\mu\\
    =&\int_B\diver(X_1)\varphi_2\,  dv_g=-\int_B\langle X_1, \nabla \varphi_2\rangle\,  dv_g\\
    =&\int_B\langle X_1, -S\nabla \varphi_2\rangle\,  d\mu'=\left\langle\mkern-4mu\left\langle (\varphi_1,X_1), \mathcal{Q}(\varphi_2,X_2) \right\rangle\mkern-4mu\right\rangle,
\end{aligned}
$$
for any $(\varphi_1,X_1),(\varphi_2,X_2)\in V$.
These operators will be called \emph{inversely weighted partner operators}, in view of the intertwining supersymmetry properties that they inherit from this decomposition (see \cite{InfeldHull1951, CooperKhareSukhatme1995, HouriSakamotoTatsumi2017}).
In view of the nature of the inversely weighted partners, the inversely weighted decomposition directly inherits the algebraic structure typically associated with supersymmetry, in particular that of a Lie superalgebra; see \cite{Scheunert1979}. 
By defining the superHamiltonian as 

$$
\mathcal{H}:V\to V,\quad \mathcal{H}(\varphi,X)=\left(\mathscr{L}_S^+(\varphi),\mathscr{L}_S^{-}(X)\right)
$$
with 
\[
\begin{aligned}
\mathscr{L}_S^{+}(\varphi)=& \mathcal{A} ^\dag_S \mathcal{A}_S(\varphi)=-\dfrac{1}{S}\diver\left({S}\nabla\varphi \right)=-\triangle_\mu\varphi, \\
\mathscr{L}_S^{-}(X)=& \mathcal{A}_S\mathcal{A}^\dag_S(X)=-S \nabla\left( \frac{\diver(X)}{S} \right) 
\end{aligned}
\]
we obtain the following relations of commutation and anti-commutation
\begin{enumerate}
    \item $\{\mathcal{Q},\mathcal{Q}^\dag\}=\mathcal{Q}\circ \mathcal{Q}^\dag+\mathcal{Q}^\dag \circ\mathcal{Q}=\mathcal{H}$,
    \item $\{\mathcal{Q},\mathcal{Q}\}=\{\mathcal{Q}^\dag,\mathcal{Q}^\dag\}=0$,
    \item $[\mathcal{H},\mathcal{Q}^\dag]=\mathcal{H}\circ\mathcal{Q^\dag}-\mathcal{Q^\dag}\circ\mathcal{H}=0$,
    \item $[\mathcal{H},\mathcal{Q}]=0$.
\end{enumerate}
Due to the construction of $\mathscr{L}_S^{+}$ and $\mathscr{L}_S^{-}$, the operators $\mathcal{A}^\dag_S$ and $\mathcal{A}_S$ satisfy an intertwining relation in the sense of linear differential operators; see \cite{HouriSakamotoTatsumi2017,KnappStein1971}. Indeed, this follows from the identities
\begin{equation*}
 \mathcal{A}_S \mathscr{L}_S^{+}
 = \mathcal{A}_S \left( \mathcal{A}_S^\dag\mathcal{A}_S \right)
 = \mathscr{L}_S^{-}\mathcal{A}_S
 \qquad \text{and} \qquad
 \mathcal{A}_S^\dag\mathscr{L}_S^{-}
 = \mathcal{A}_S^\dag\left( \mathcal{A}_S \mathcal{A}^\dag_S \right)
 = \mathscr{L}_S^{+}\mathcal{A}_S^\dag.
\end{equation*}

The first natural question in this context is to understand how the Dirichlet eigenvalue problem of the operator $\mathscr{L}_S^{+}$ is transferred to the operator $\mathscr{L}_S^{-}$ and, in particular, how this procedure affects the boundary conditions of the problem.

\begin{theorem}\label{DirichletNeumannProblem}
Let $\Omega\subset B$ be a precompact domain with smooth boundary $\partial \Omega$. Consider the Dirichlet and divergence-free eigenvalue problems for the inversely weighted partners $\mathscr{L}_S^{+}$ and $\mathscr{L}_S^{-}$:
\begin{subequations}\label{DN}
\begin{equation}\label{DN-a}
\left\{
\begin{aligned}
\mathscr{L}_S^{+}\varphi &= \lambda \varphi \quad &&\text{in } \Omega
,\\
\varphi &= 0 \quad &&\text{on } \partial \Omega,
\end{aligned}
\right.
\end{equation}

\begin{equation}\label{DN-b}
\left\{
\begin{aligned}
\mathscr{L}_S^{-}X &= \lambda X \quad &&\text{in } \Omega,\\
\diver(X) &= 0 \quad &&\text{on } \partial \Omega.
\end{aligned}
\right.
\end{equation}
\end{subequations}
Then the following statements hold:
\begin{enumerate}
 \item If $\lambda \neq 0$ and $\varphi \in C^{\infty}(B)$ is a nontrivial solution of the Dirichlet eigenvalue problem for $\mathscr L_S^{+}$, then
 \[
 X \coloneqq \mathcal A_S \varphi = S \nabla \varphi 
 \]
is a nontrivial solution of the divergence-free boundary problem associated with $\mathscr L_S^{-}$ in \eqref{DN-b}, with the same eigenvalue $\lambda$.
 \item If $\lambda \neq 0$ and $X \in \mathfrak{X}(M)$ is a nontrivial solution of the divergence-free boundary problem for $\mathscr L_S^{-}$, then
 \[
 \varphi:=\mathcal{A}^\dag_SX=\frac{1}{S}\diver(X)
 \]
 is a nontrivial solution of the Dirichlet problem associated with $\mathscr L_S^{+}$ in \eqref{DN-a}, with the same eigenvalue $\lambda$.
\end{enumerate}
Consequently, the nonzero Dirichlet spectrum of $\mathscr L_S^{+}$ coincides with the nonzero divergence-free boundary spectrum of $\mathscr L_S^{-}$, counting multiplicities. 
\end{theorem}

\begin{proof}
Assume first that $\varphi$ is a nontrivial solution of \eqref{DN-a} with $\lambda \neq 0$, and define $X \coloneq \mathcal A_S\varphi$. By the intertwining relation
\[
\mathcal A_S\mathscr L_S^{+} = \mathscr L_S^{-}\mathcal A_S,
\]
this yields
\[
\mathscr L_S^{-}X
= \mathscr L_S^{-}(\mathcal A_S\varphi)
= \mathcal A_S(\mathscr L_S^{+}\varphi)
= \lambda \mathcal A_S\varphi
= \lambda X.
\]
It remains to determine the effect on the boundary condition. Since $\diver(X)= \diver(S\nabla\varphi)$ and
\[
\mathscr{L}_S^{+} \varphi \coloneqq \dfrac{1}{S}\diver\left({S}\nabla\varphi \right),
\]
it follows that
\[
\diver(X) = S\mathscr L_S^{+}\varphi = \lambda S\varphi.
\]
Since $S$ is continuous in $\Omega$ and $\varphi|_{\partial \Omega}=0$, we conclude that $\diver(X)=0$ on $\partial \Omega$. Hence, $X$ satisfies the divergence-free boundary problem \eqref{DN-b} for $\mathscr L_S^{-}$. Moreover, if $X \equiv 0$, then $S \nabla\varphi\equiv 0$, and since $S>0$, therefore $\nabla \varphi\equiv 0$. Thus, $\varphi$ is constant in each connected component. Hence, by the Dirichlet boundary condition $\varphi = 0$, a contradiction. Therefore $X \not\equiv 0$.

Conversely, now assume that $X$ is a nontrivial solution of the divergence-free boundary problem \eqref{DN-b} with $\lambda \neq 0$, and define $\varphi \coloneq \mathcal{A}^\dag_SX$. Using again the intertwining relation
\[
\mathcal{A}^\dag_S\mathscr{L}_S^{-} = \mathscr{L}_S^{+}\mathcal{A}^\dag_S,
\]
we conclude
\[
\mathscr L_S^{+}\varphi = \mathscr L_S^{+}(\mathcal A^\dag_SX) = \mathcal A^\dag_S(\mathscr L_S^{-}X) =\lambda \mathcal{A}^\dag_SX = \lambda\varphi.
\]

Furthermore, $\varphi|_{\partial \Omega}= \mathcal{A}^\dag_S (X)|_{\partial \Omega} = (1/S) \diver(X)|_{\partial \Omega} = 0$, so $\varphi$ satisfies the Dirichlet problem \eqref{DN-a}. Finally, if $\varphi = 0$, then $\diver(X)=0$. Hence,
\[
0= \nabla\left( \diver(X) \right) - \diver(X) \nabla \log S =\mathscr L_S^{-}X= \lambda X.
\]
Since $\lambda \neq 0$, this implies that $X \equiv 0$, again a contradiction. Therefore $\varphi \neq 0$.

It follows that the maps
\[
X\mapsto \mathcal A_S^\dag X,
\qquad
\psi\mapsto \mathcal A_S\psi,
\]
define mutually inverse correspondences between the eigenspaces associated with each nonzero eigenvalue $\lambda$. This proves the claimed spectral correspondence.
\end{proof}

\begin{remark}
Observe that the first eigenvalue of the divergence-free boundary problem for $\mathscr L_S^{-}$ in \eqref{DN-b} is zero, with the associated eigenspace consisting of solenoidal fields. Apart from this zero mode, Theorem \ref{DirichletNeumannProblem} shows that the Dirichlet eigenpairs $(\varphi_k,\lambda_k^{\mathcal D})$ associated with \eqref{DN-a} are mapped by $\mathcal A_S$ to the eigenpairs $(X_k,\lambda_k^{\text{Div}})$ associated with the divergence-free boundary problem \eqref{DN-b}, for $k\geq 1$, as illustrated in the diagram below (see Figure \ref{fig:DNcorresfig}).
\end{remark}

\begin{figure}[htb!]
 \centering
 		\begin{tikzpicture}[scale=1.0,>=stealth]
			
				\def\u{5}
				\def\v{1.5}	
				
				\draw[line width=1pt] (-\u,0) node[below] {$\varphi_{k}$} to (\u,0) node[below] {$X_{k+1}$};
				\fill[white] (0,0.4) arc (90:360:1cm and 0.4cm);
				\draw[<-,line width=1pt,fill=white] (0,0.4) arc (90:270:1cm and 0.4cm);
				\draw[->,line width=1pt,fill=white!] (0,0.4) arc (90:-90:1cm and 0.4cm);
				\node[below] (t1) at (0,-0.5) {$\mathcal{A}_{S}$};
				\node[above] (t2) at (0,0.5) {$\mathcal{A}^\dag_{S}$};
				
				\fill (3,-0.5-0.5) circle (2pt);
				\fill (3,-1-0.5) circle (2pt);
				\fill (3,-1.5-0.5) circle (2pt);
				
				\fill (-3,-0.5-0.5) circle (2pt);
				\fill (-3,-1-0.5) circle (2pt);
				\fill (-3,-1.5-0.5) circle (2pt);

				\draw[line width=1pt] (-\u,-3) node[below] {$\varphi_{1}$} to (\u,-3) node[below] {$X_{2}$};

				\fill[white] (0,0.4-3) arc (90:360:1cm and 0.4cm);
				\draw[<-,line width=1pt,fill=white] (0,0.4-3) arc (90:270:1cm and 0.4cm);
				\draw[->,line width=1pt,fill=white!] (0,0.4-3) arc (90:-90:1cm and 0.4cm);
				\node[below] (t3) at (0,-0.5-3) {$\mathcal{A}_{S}$};
				\node[above] (t4) at (0,0.5-3) {$\mathcal{A}^\dag_{S}$};

				\draw[line width=1pt] (0,-6) to (\u,-6) node[below] {$X_{0}$};
				
				\node[above] (t5) at (2.5,-6) {Divergence-free zero mode};

		\end{tikzpicture}
 \caption{Dirichlet and divergence-free correspondence between eigenpairs.}
 \label{fig:DNcorresfig}
\end{figure}
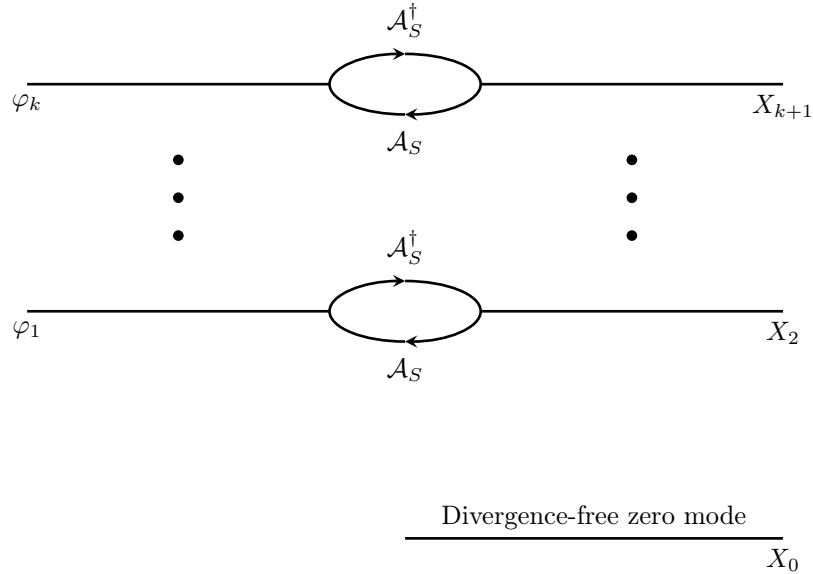

The correspondence established above between the Dirichlet and divergence-free boundary conditions allows us to state the following characterization of the first eigenvalue for precompact domains.

\begin{corollary}\label{ElCorollari}
 Let $(B,g_B)$ be a Riemannian manifold with weight $d\mu= Sdv_g$. Let $\Omega \subset B$ be a precompact domain with smooth boundary. Then the first (positive) eigenvalue of 
 $$
-\triangle_\mu \phi=\lambda \phi\quad {\rm in}\quad \Omega,
 $$
 with Dirichlet boundary condition on $\partial \Omega$ is given by
 $$
\lambda_1^\mu(\Omega)=\min_{X\in\mathfrak{sol}^\perp(\Omega)} \left\{\frac{\displaystyle \int_\Omega \left(\frac{\diver(X)}{S}\right)^2d\mu'}{\displaystyle\int_\Omega \Vert X\Vert^2d\mu'}\right\}.
 $$
 Where $d\mu'=\frac{1}{S}dv_g$ and $\mathfrak{sol}^\perp(D)$ is the orthogonal complement of the solenoidal vector fields with respect to the norm
$$
(-,-):\mathfrak{X(}D)\times\mathfrak{X(}D)\to \mathbb{R}, \quad (X,Y)=\int_D\langle X,Y\rangle d\mu', 
$$
and with free-divergence boundary condition.
\end{corollary}
\begin{proof}
    Observe first of all that $\mathscr{L}_S^-$ is self-adjoint in the following sense
    $$
    \begin{aligned}
\int_\Omega\langle Y,\mathscr{L}_S^-(X)\rangle d\mu'= &\left\langle\mkern-4mu\left\langle (0,Y), \mathcal{Q}\mathcal{Q}^\dag(0,X) \right\rangle\mkern-4mu\right\rangle=\left\langle\mkern-4mu\left\langle \mathcal{Q}^\dag(0,Y), \mathcal{Q}^\dag(0,X) \right\rangle\mkern-4mu\right\rangle\\
=&\left\langle\mkern-4mu\left\langle \mathcal{Q}\mathcal{Q}^\dag(0,Y), (0,X) \right\rangle\mkern-4mu\right\rangle=\int_\Omega\langle \mathscr{L}_S^-(Y),X\rangle d\mu',
    \end{aligned}
    $$
    for any $X,Y\in \mathfrak{X}(B)$ satisfying the free-divergence boundary condition on $\partial \Omega$. Therefore, we can find an orthonormal eigenbasis $\{X_i\}$ such that
    $$
\mathscr{L}_S^-(X_i)=\lambda_iX_i,\quad \int_\Omega\langle X_i,X_j\rangle d\mu'=\delta_{ij} .
    $$
    Since the case $\lambda_0=0$ corresponds to solenoidal vector fields and $0<\lambda_1=\lambda_1^{\mu}(\Omega)$, a standard Rayleigh quotient argument (see \cite{Chavel}) yields:
    $$
    \begin{aligned}
\lambda_1^{\mu}(\Omega)=\min_{X \in \mathfrak{sol}^\perp(\Omega)} \frac{\int_\Omega\langle X,\mathscr{L}_S^-(X)\rangle d\mu'}{\int_\Omega\langle X,X\rangle d\mu'}=&\min_{X \in \mathfrak{sol}^\perp(\Omega)}\frac{\left\langle\mkern-4mu\left\langle \mathcal{Q}^\dag(0,X), \mathcal{Q}^\dag(0,X) \right\rangle\mkern-4mu\right\rangle}{\int_\Omega\Vert X\Vert^2 d\mu'}\\
=&\min_{X \in \mathfrak{sol}^\perp(\Omega)} \frac{\displaystyle \int_\Omega \left(\frac{\diver(X)}{S}\right)^2d\mu'}{\displaystyle\int_\Omega \Vert X\Vert^2d\mu'}\cdot        
    \end{aligned}
$$
\end{proof}

%%%%%%%%%%%%%%%%%%%%%%%%%%%%%%%%%%%%%%%%%%%%%%%%%%%%%%%%%%%%%%%%%%%%%%%%%%%%%%%%%%%%%%%%%%%%%%%%%%%%%%%%%%%%%%%%%%%%%%%%%%%%%%%%%%%%%%%%%%%%%%%%
\section{The case of one-dimensional base manifold. Proof of Theorem \ref{DNcorres}, proof of corollary \ref{PAYNE}, proof of Theorem \ref{teo:sum} and corollary \ref{cor:sumneumann}.}\label{sec:sum} 
%%%%%%%%%%%%%%%%%%%%%%%%%%%%%%%%%%%%%%%%%%%%%%%%%%%%%%%%%%%%%%%%%%%%%%%%%%%%%%%%%%%%%%%%%%%%%%%%%%%%%%%%%%%%%%%%%%%%%%%%%%%%%%%%%%%%%%%%%%%%%%%%

As we have seen, the basic spectrum of a manifold which admits a Riemannian submersion with fibers having a basic mean curvature vector field reduces to the problem of finding the spectrum on the base manifold of a weighted Laplacian, where the weight is precisely the volume of the fibers at each point of the base manifold. This spectrum of the weighted Laplacian is related via the inverse weighted decomposition to an eigenvalue problem for vector fields. 

In this section, we show that this duality becomes more transparent when the base manifold is one-dimensional, and we exhibit the correspondence $S\mapsto \frac{1}{S}$. Moreover, we present an independent result showing that the sum of the reciprocal eigenvalues of a weighted Laplacian admits a closed formula for the Dirichlet and Neumann boundary conditions.

\subsection{The inverse weighted decomposition in the one-dimensional case. Proof of Theorem \ref{DNcorres}}
In the case of a one-dimensional connected manifold $B$, $B$ is isometric to some (rescaled if needed) $\mathbb S^1$ if $B$ is compact without boundary, to the real line $\mathbb{R}$ if $B$ is non-compact without boundary and of infinite measure, or to a connected subset of $\mathbb R$ otherwise. In every case, we can choose a unit vector field $\partial t$ such that 
$$
\mathfrak{X}(B)=\left\lbrace f\,\partial t\, :\, f\in C^\infty(B) \right\rbrace.
$$
This identification is possible because since $B$ is one-dimensional, $B$ is orientable. It allows us to identify $\mathfrak{X}(B)\simeq C^\infty(B)$ through the maps
$$
X\mapsto f=\langle X,\partial t\rangle,\quad f\mapsto X=f\, \partial t.
$$
With this identification, the operators $\mathcal A_S$ and $\mathcal A_S^\dag$ act on $C^\infty$ as 
$$
\mathcal{A}_S(\phi)=S(t) \phi'(t)\partial t, \quad \mathcal A^\dag_S (\phi \partial t)=\frac{1}{S(t)}\phi'(t).
$$
With some abuse of notation, we can write in the one-dimensional case 
\begin{equation*}
\begin{array}{ccc}
 \mathcal{A^\dag}_S \psi \coloneq S\psi' & \text{and}& \mathcal{A}_S \psi \coloneq \frac{1}{S}\psi',\\
 \mathscr{L}_S^+(\psi)=\frac{1}{S}(S\psi')' & \text{and}&\mathscr{L}_S^-(\psi)=S\left(\frac{\psi'}{S}\right)'.
\end{array}
\end{equation*}
Hence, we obtain the main result for the one-dimensional base manifold:
$$
 \mathscr{L}_S^+= \mathscr{L}_{\frac{1}{S}}^-.
$$
In this setting, the divergence-free boundary condition $\diver(X)=0$ on $\partial\Omega$ becomes, for a vector field $X = f \partial t$,
\[
\diver(X)= f' = 0 \quad \text{on } \partial\Omega,
\]
which is exactly the Neumann condition for the associated scalar function $f$. 
If $\Omega=\pi^{-1}([a,b])$, since the base manifold is one-dimensional, $\Omega$ is diffeomeorfic to
$$
\Omega\simeq [a,b]\times F,\quad F:=\pi^{-1}(a),
$$
and $\Omega$ satisfies the boundary condition
$$
\partial \Omega=\pi^{-1}(a)\cup\pi^{-1}(b).
$$

Thus, Theorem \ref{DirichletNeumannProblem} directly implies Theorem \ref{DNcorres}. Note that the same algebraic relation $\mathscr{L}_S^+= \mathscr{L}_{1/S}^-$ holds for the operators, and the transformation $S\mapsto 1/S$ exchanges the roles of Dirichlet and Neumann spectra.
\subsection{Proof of Corollary \ref{PAYNE}} For the proof of item (1), we first observe that if $(\log S(t))''\geq 0$, then the Schr\"{o}dinger potentials generated by the fiber-volume function,
as given in \eqref{SDcorres}, satisfy
\begin{equation}\label{ine}
\begin{aligned}
\mathcal{V}_{1/S}
&= \frac{1}{4}\left[(\log S(t))'\right]^2
-\frac{1}{2}(\log S(t))'' \\
&\leq
\frac{1}{4}\left[(\log S(t))'\right]^2
+\frac{1}{2}(\log S(t))'' \\
&= \mathcal{V}_{S}.
\end{aligned}
\end{equation}

By the unitary equivalence between $\triangle|_{\mathrm{basic}}$ and
$\mathscr{H}_S$, we have $\lambda^{\mathcal D}_{1,\mathrm{basic}}(\Omega)
= \lambda^{\mathcal D}_{1,\mathscr{H}_S}((a,b))$. Moreover, \eqref{ine} and the min-max principle imply
\[
\lambda^{\mathcal D}_{1,\mathscr{H}_{1/S}}((a,b))
\leq
\lambda^{\mathcal D}_{1,\mathscr{H}_{S}}((a,b)).
\]
The inequality $\lambda_{2,{\text{basic}}}^{\mathcal{N}}(\Omega) \leq \lambda_{1, \text{basic}}^{\mathcal{D}}(\Omega)$ follows directly from Theorem \ref{DNcorres}. Since the first basic Dirichlet eigenfunction associated with $\lambda_{1, \text{basic}}^{\mathcal{D}}(\Omega)$ does not change sign, we have
$$
\lambda_{1}^{\mathcal{D}}(\Omega)=\lambda_{1, \text{basic}}^{\mathcal{D}}(\Omega).
$$
Furthermore, since it is clear that $\lambda_{2}^{\mathcal{N}}(\Omega)\leq \lambda_{2,{\text{basic}}}^{\mathcal{N}}(\Omega)$, we consequently obtain
$$
\lambda_{2}^{\mathcal{N}}(\Omega)\leq \lambda_{1}^{\mathcal{D}}(\Omega).
$$
Case (2) follows by observing that, if $(\log S(t))''\leq 0$, then the
inequality in \eqref{ine} holds in the reverse direction. For the equality case, assume first that $(\log S(t))''=0$ on $[a,b]$.
Then both assertions (1) and (2) apply, and therefore
$
\lambda_{2,\mathrm{basic}}^{\mathcal N}(\Omega)
\leq
\lambda_{1,\mathrm{basic}}^{\mathcal D}(\Omega)
\leq
\lambda_{2,\mathrm{basic}}^{\mathcal N}(\Omega).
$
Consequently, $\lambda_{2,\mathrm{basic}}^{\mathcal N}(\Omega) =
\lambda_{1,\mathrm{basic}}^{\mathcal D}(\Omega)$. Finally, suppose that we have $\lambda_{2,\mathrm{basic}}^{\mathcal N}(\Omega) = \lambda_{1,\mathrm{basic}}^{\mathcal D}(\Omega)$. Let $\varphi_S>0$ be the first normalized Dirichlet eigenfunction of
$\mathscr H_S$ on $(a,b)$, that is, $\|\varphi_S\|_{L^2(a,b)}^2=1$.
Then
\[
\lambda_{1,\mathscr H_S}^{\mathcal D}((a,b)) = \int_a^b
\left( |\varphi_S'(t)|^2+\mathcal V_S(t)\varphi_S^2(t) \right)dt.
\]
Using $\varphi_S$ as a test function in the Rayleigh quotient for
$\mathscr H_{1/S}$, we obtain
\begin{align*}
\lambda_{1,\mathscr H_{1/S}}^{\mathcal D}((a,b))
&\leq \int_a^b \left( |\varphi_S'(t)|^2+\mathcal V_{1/S}(t)\varphi_S^2(t)
\right)dt \\
&= \int_a^b \left( |\varphi_S'(t)|^2+\mathcal V_S(t)\varphi_S^2(t) \right)dt
- \int_a^b (\log S(t))''\varphi_S^2(t)dt \\
&= \lambda_{1,\mathscr H_S}^{\mathcal D}((a,b)) - \int_a^b (\log S(t))''\varphi_S^2(t)dt.
\end{align*}
Under the assumptions of log-convexity and log-concavity, the result follows.

\subsection{On the reciprocal sum of basic eigenvalues with Dirichlet boundary conditions. Proof of Theorem \ref{teo:sum} and Corollary \ref{cor:sumneumann}}
Given the interval $[a,b]\subset \mathbb{R} $ with weight $\mu$, the solution of the following Poisson problem with $f\in L^2([a,b],\mu)$
$$
\left\lbrace 
\begin{array}{rllc}
\triangle_\mu\phi &=-f&{\rm in}&[a,b] \\
 \phi&=0 &{\rm in}&\{a,b\}, 
\end{array}
\right.
$$
is given by 
$$
\phi(x)=\int_a^bf(y)g(x,y)d\mu(y),
$$
where $g$ is the Green function with Dirichlet boundary values. Conversely, by Theorem 6.1 of \cite{bessa2016}, the sum of the reciprocals of the Dirichlet eigenvalues for the weighted Laplacian, in this one-dimensional case, can be obtained by the trace of the Green function as
 $$
\sum_{k=1}^{\infty}\frac{1}{\lambda_k^{\mathcal{D}}(\Omega)}=\int_{[a,b]}g(x,x)d\mu(x).
$$

The Green function is bounded in terms of capacity by Proposition 4.1 of \cite{GriExp} as follows.

\begin{lemma}\label{grygcap}
Let $\Omega$ be a precompact set and let $U$ be an open precompact set such that $\overline U\subset \Omega$. Then for any $y\in U$,
 $$
\inf_{x\in \partial U} g_\Omega(x,y)\leq {\rm cap}(U, \Omega)^{-1}\leq \sup_{x\in \partial U} g_\Omega(x,y). 
 $$
\end{lemma}

Now set $y\in (a,b)$ and set $U=(y-\epsilon,y+\epsilon)$. Then by the previous lemma,
\[
\min_{x\in \{y-\epsilon,y+\epsilon\} } g(x,y)\leq {\rm cap}((y-\epsilon,y+\epsilon), [a,b])^{-1}\leq \max_{x\in \{y-\epsilon,y+\epsilon\} } g(x,y). 
\]
Given a precompact domain $G\subset B$ and a closed subset $F\subset G$, the weighted capacity ${\rm cap}(F,G)$ is given by
\[
{\rm cap}(F,G)=\int_{G\setminus F}\Vert \nabla^B \widetilde \Psi\Vert^2d\mu,
\]
with $\widetilde\Psi$ satisfying the following problem:
\begin{equation}\label{eq:1.6}
 \left\{
 \begin{aligned}
 \triangle_\mu\widetilde\Psi&=0& {\rm in }\, &G\\
 \widetilde\Psi(x)&=1& {\rm if }\, &x\in F\\
 \widetilde\Psi(x)&=0& {\rm if }\, &x\in \partial G
 \end{aligned}
 \right.
\end{equation}
A standard computation shows that
\[
 {\rm cap}((y-\epsilon,y+\epsilon), [a,b])=\left(\int_a^{y-\epsilon}\frac{dt}{S(t)}\right)^{-1}+\left(\int_{y+\epsilon}^b\frac{dt}{S(t)}\right)^{-1}.
\]
Hence,
\[
\min_{x\in \{y-\epsilon,y+\epsilon\} } g(x,y)\leq \frac{1}{\left( \displaystyle \int_a^{y-\epsilon}\frac{dt}{S(t)}\right)^{-1}+\left( \displaystyle\int_{y+\epsilon}^b\frac{dt}{S(t)}\right)^{-1}}\leq \max_{x\in \{y-\epsilon,y+\epsilon\} } g(x,y). 
\]
Taking now the limit $\epsilon\to 0$ and using the continuity of $g(\cdot,\cdot)$ away from the diagonal (and the arguments of standard potential theory), we obtain the following result
\begin{equation}\label{eq:3.2}
g(y,y)= \frac{1}{\left( \displaystyle \int_a^{y}\frac{dt}{S(t)}\right)^{-1}+\left( \displaystyle \int_{y}^b\frac{dt}{S(t)}\right)^{-1}}=\frac{\left(\displaystyle \int_a^{y}\frac{dt}{S(t)}\right)\left( \displaystyle\int_y^{b} \frac{dt}{S(t)}\right)}{\displaystyle\int_a^{b}\frac{dt}{S(t)}} \cdot 
\end{equation}
Therefore,
\begin{equation}\label{eq:basicespectrum}
\begin{aligned}
 \sum_{k=1}^{\infty}\frac{1}{\lambda_k^{\mathcal{D}}(\Omega)}=&\int_{[a,b]}g(y,y)d\mu(y)=\int_{[a,b]}g(y,y)S(y)dy \\
 =&\left( {\int_{a}^b\frac{dt}{S(t)}} \right)^{-1}\int_a^b S(x)\left(\int_a^x\frac{dt}{S(t)}\int_x^b\frac{dt}{S(t)}\right)dx.
\end{aligned}
\end{equation}
Taking into account the boundary conditions and that the basic Dirichlet spectrum coincides with the Dirichlet spectrum on the base manifold, this finishes the proof of Theorem \ref{teo:sum}. Finally, using the transformation $S\mapsto\frac{1}{S}$, the Corollary \ref{cor:sumneumann} is proved.

%%%%%%%%%%%%%%%%%%%%%%%%%%%%%%%%%%%%%%%%%%%%%%%%%%%%%%%%%%%%%%%%%%%%%%%%%%%%%%%%%%%%%%%%%%%%%%%%%%%%%%%%%%%%%%%%%%%%%%%%%%%%%%%%%%%%%%%%%%%%%%%%
\section{Classical examples of reciprocal basic eigenvalue sums and comparison with constant fiber growth manifolds} \label{Seccomparison}
%%%%%%%%%%%%%%%%%%%%%%%%%%%%%%%%%%%%%%%%%%%%%%%%%%%%%%%%%%%%%%%%%%%%%%%%%%%%%%%%%%%%%%%%%%%%%%%%%%%%%%%%%%%%%%%%%%%%%%%%%%%%%%%%%%%%%%%%%%%%%%%%

This section is dedicated to the computation of the reciprocal sum of the basic eigenvalues of precompact domains in the catenoid and in the pseudosphere. By analyzing a suitable precompact domain in the pseudosphere, we recover, through our framework, a generalized Basel-type summation formula. Furthermore, we establish a comparison criterion with manifolds having constant fiber volume, which provides upper and lower bounds for any Riemannian submersion over an interval with compact fibers of basic mean curvature.

\begin{example}
    We now investigate the reciprocal sum of the basic eigenvalues in several classical models. We first estimate the reciprocal sum of the basic eigenvalues for the symmetric precompact domains in the catenoid and a suitable precompact domain in the pseudosphere. Moreover, we compute the corresponding quantity for spaces with this type of growth in fiber volume.

    \begin{enumerate}
        \item\emph{Truncated catenoid}. The catenoid $\mathcal{C} \subset \mathbb{R}^3$ is a minimal surface of revolution. It can be described by the rotation around the $z$-axis of the catenary curve given by $\alpha(v) = (a \cosh(v/a),0,v)$ where $a>0$ is a constant. Therefore, the truncated catenoid can be parameterized by $\mathbf{x}: [0, 2\pi] \times[-\ell,\ell] \to \mathcal{C}$ 

\begin{equation*}
\mathbf{x}(u,v)= \left( a\cosh{\left( \dfrac{v}{a} \right)\cos{u}}, a\cosh{\left( \dfrac{v}{a} \right)\sin{u}}, v \right).
\end{equation*}
A direct computation gives
\begin{equation*}
    \begin{aligned}
         \mathbf{x}_u(u,v) &= \left( -a \cosh{\left( \frac{v}{a} \right)}\sin{u}, a \cosh{\left( \frac{v}{a} \right)}\cos{u},0 \right), \\
\mathbf{x}_v(u,v)&= \left( \sinh{\left( \frac{v}{a} \right)}\cos{u},\sinh{\left( \frac{v}{a} \right)}\sin{u}, 1 \right).
    \end{aligned}
\end{equation*}
Hence, we obtain the following standard relations
\[
\langle \mathbf{x}_u, \mathbf{x}_u \rangle = a^2\cosh^2{\left( \frac{v}{a} \right)}, \quad \langle \mathbf{x}_v, \mathbf{x}_v \rangle = \cosh^2{\left( \frac{v}{a} \right)}, \quad \langle \mathbf{x}_u, \mathbf{x}_v \rangle = 0.
\]

Therefore, the metric $g_\mathcal{C}$ induced on the catenoid is given by
\begin{equation*}
g_\mathcal{C} = \cosh^2{\left( \frac{v}{a} \right)}dv^2 + a^2\cosh^2{\left( \frac{v}{a} \right)}du^2.
\end{equation*}
To put the metric in warped-product form, we introduce the arc-length
coordinate along the meridian direction:
\[
t = \int_0^v \cosh{\left( \frac{s}{a} \right)}ds = a\sinh{\left( \frac{v}{a} \right)}, 
\quad dt = \cosh{\left( \frac{v}{a} \right)}dv.
\]
Moreover, since $r/a = \sinh{(v/a)}$, we have
\[
\cosh^2{\left( \frac{v}{a} \right)} = 1+ \sinh^2{\left( \frac{v}{a} \right)} = 1+ \left( \frac{t}{a} \right)^2.
\]

Thus, substituting the above change of variables into metric equation above, we find that the truncated catenoid is isometric to the warped product $(\mathcal C,g_{\mathcal C})
\cong [-L,L]\times_{\sqrt{a^2+t^2}}\mathbb S^1$ where $L = a \sinh(\ell/a)$. Since the warping function is given by $w(t) = \sqrt{a^2+t^2}$, by item (1) in Example \ref{expwap} we see that $\mathcal{C}$ admits a Riemannian submersion $\pi: \mathcal{C} \to [-L,L]$ with compact fibers of basic mean curvature and fiber-volume function given by $S(t) = 2\pi \sqrt{a^2 + t^2}$ as illustrated in Figure \ref{fig:catenoid}

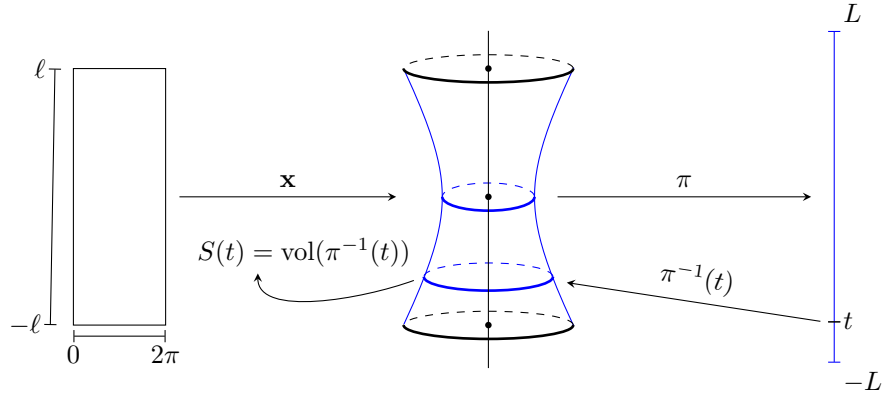
\begin{figure}[htb!]
\hfill
 \def\r{7}
	\begin{tikzpicture}[scale=0.61,>=stealth,y=0.6cm]
		\def\u{1}
		\def\v{0.3}
		\def\t{1.7}
		\coordinate (p0) at (0:0cm);
		\coordinate (p2) at ($(p0)+(\u,0)$);
		
		\fill (p0) circle (2pt);
		\draw[dashed,blue] (p2) arc (0:180:1cm and \v cm);
		\draw[line width=1pt,blue] (p2) arc (360:180:1cm and \v cm);

		\draw[domain=-1:1, blue] plot ({\u*pow(cos(\x r),-1)},{3*tan(\x r)});
		
		\draw[domain=-4.135:-2.15, blue] plot ({\u*pow(cos(\x r),-1)},{3*tan(\x r)});
		
		\coordinate (p3) at ($(p0)+(0,-0.7*\t)$);
		
		\draw[line width=1pt, blue] ($(p3)+(1.4,-\t)$) arc (360:180:1.4cm and \v cm);
		
		\draw[dashed, blue] ($(p3)+(1.4,-\t)$) arc (0:180:1.4cm and \v cm);
		
		\draw ($(p0)+(0,6)$) -- ($(p3)+(0,-5)$);
		
		\coordinate (p4) at ($(p0)+(0,-2.727*\t)$);
		
		\fill (p4) circle (2pt);
		
		\draw[line width=1pt] ($(p4)+(1.84,0)$) arc (360:180:1.84cm and \v cm);
		
		\draw[dashed] ($(p4)+(1.84,0)$) arc (0:180:1.84cm and \v cm);
		
		\coordinate (p5) at ($(p0)+(0,2.727*\t)$);
		
		\fill (p5) circle (2pt);
		
		\draw[line width=1pt] ($(p5)+(1.84,0)$) arc (360:180:1.84cm and \v cm);
		
		\draw[dashed] ($(p5)+(1.84,0)$) arc (0:180:1.84cm and \v cm);
		
		\begin{scope}[xshift=-9.0cm]
			\draw (0,2.727*\t) rectangle ($(0,-2.727*\t)+(2,0)$);
			
			\draw[|-|] (-0.4,2.727*\t) node[left] {$\ell$}-- ($(-0.5,-2.727*\t)$) node[left] {$-\ell$};
			
			\draw[|-|] (0,-2.727*\t-0.4) node[below] {0}-- (2,-2.727*\t-0.4) node[below] {$2\pi$};
			
		\end{scope}

		\begin{scope}[xshift=7.5cm]
			\draw[|-|, blue] (0,3.527*\t) node[above right] {$\textcolor{black}{L}$} -- ($(0,-3.527*\t)$) node[below right] {$\textcolor{black}{-L}$};

 \draw[<-] (-5.8,-3.1) -- (-0.3,-4.5)
 node[pos=0.5, sloped, above] {$\pi^{-1}(t)$};

			\draw[|-|] (0,-4.5) node[right] {$t$};
		\end{scope}
		
		\draw[->] (-6.7,0) -- (-2,0) node[pos=0.5,above] {$\mathbf{x}$};
		
		\draw[->] (1.5,0) -- (7,0) node[pos=0.5,above] {$\pi$};
		
		\node (t1) at (-4,-2) {$S(t) = {\rm vol} (\pi^{-1}(t))$};
		\draw[->,shorten <=4pt,shorten >=8pt] ($(p3)+(-1.4,-\t)$) to[out=200,in=-90] (-5,-2);		
	\end{tikzpicture}
    \caption{Fibration of the catenoid $\mathcal{C}$.}
 \label{fig:catenoid}
\end{figure}
We now apply Theorem \ref{teo:sum} to compute the reciprocal sum of the basic spectrum of the truncated catenoid. We define the function 
\[ \theta(x) \coloneq \operatorname{arcsinh}{\left( \frac{x}{a} \right)}\]
and consider the following computation:
\begin{equation*}
 \begin{aligned}\label{Sintegral}
 \int_{-L}^L \frac{1}{S(t)}dt =& \frac{1}{2\pi}\int_{-L}^{L} \frac{1}{\sqrt{a^2 + t^2}} dt \\
 =& \frac{1}{\pi } \theta(L).
 \end{aligned}
\end{equation*}
Using the identity above once again, together with the fact that $\theta$ is an odd function, we have the following
\begin{equation*}
\begin{aligned}\label{integral}
 \int_{-L}^x \frac{1}{S(t)}dt &= \frac{1}{2\pi} \left[ \theta(x) + \theta(L) \right], \\
\int_{x}^{L} \frac{1}{S(t)}dt &= \frac{1}{2\pi} \left[ \theta(L) - \theta(x) \right]. 
\end{aligned}
\end{equation*}
Multiplying the given identities yields 
\[ \left( \int_{-L}^x \frac{1}{S(t)}dt \right) \left( \int_{x}^{L} \frac{1}{S(t)}dt \right) = \frac{1}{4\pi^2} \left[ \theta^2(L) - \theta^2(x) \right].\]
Since $S(x)=2\pi\sqrt{a^2+x^2}$, using the identity above and the change of variables $x=a\sinh\theta$, yields the following computation: 
\begin{equation*}
\begin{aligned}\label{SP}
     \frac{\pi}{a^2}\int_{-L}^LS(x) \left( \left( \int_{-L}^x \frac{1}{S(t)}dt \right) \left( \int_{x}^{L} \frac{1}{S(t)}dt \right) \right)dx &= \frac{1}{3}\theta^3(L)\\ &+ \frac{1}{4}\theta(L)\cosh{\left( 2 \theta(L) \right)} \\
     &-\frac{1}{8}\sinh{\left( 2\theta(L) \right)}.
\end{aligned}
\end{equation*}
Therefore, substituting \eqref{Sintegral} and the identity above into the formula in Theorem \ref{teo:sum}, we obtain
\begin{equation*}\label{boundedspec}
 \begin{aligned}
  \sum_{k=1}^{\infty}\frac{1}{\lambda_k^{\mathcal{D}}(\mathcal{C})} =& \left( \int_{-L}^L\frac{dt}{S(t)} \right)^{-1} \int_{-L}^L 
 S(x)\left(\int_{-L}^x\frac{dt}{S(t)}\int_x^L\frac{dt}{S(t)}\right)dx \\
 =& a^2 \left[ \frac{1}{3}\theta^2(L) +\frac{1}{4}\cosh{\left( 2 \theta(L) \right)} -\frac{1}{8\theta(L)}\sinh{\left( 2\theta(L) \right)}  \right].
 \end{aligned}
\end{equation*}

\item \emph{Pseudosphere}. We consider the \emph{pseudosphere} $\mathcal{S}$, or tractricoid, obtained by rotating a
tractrix around the $z$-axis. Let $\gamma(v) = \left( e^{-\alpha v},0,\psi(v) \right)$ be the generating curve, where
\[
\psi(v) = \frac{1}{\alpha} \left[ \sqrt{ 1- \alpha^2 e^{-2\alpha v}}
+ \log \left( \frac{\alpha e^{-\alpha v}}{1+\sqrt{1-\alpha^2 e^{-2\alpha v}}}\right) \right].
\]
The corresponding surface of revolution is therefore parametrized by
\[
\mathbf{y}(u,v)
= \left( e^{-\alpha v}\cos u, e^{-\alpha v}\sin u, \psi(v) \right),
\qquad v\in[0,2\pi).
\]
It is a classical model surface with constant negative Gaussian curvature  \(K=-\alpha^2\). A direct computation gives
\[
\langle \mathbf{y}_u, \mathbf{y}_u \rangle = e^{-2\alpha v}, \qquad \langle \mathbf{y}_v, \mathbf{y}_v \rangle = 1, \qquad \langle \mathbf{y}_u, \mathbf{y}_v \rangle = 0.
\]

Therefore, the metric $g_\mathcal{S}$ induced on the pseudosphere is given by
\begin{equation*}
g_\mathcal{S} = dv^2 + e^{-2\alpha v}du^2. 
\end{equation*}
Thus, the pseudosphere $(\mathcal{S},g_\mathcal{S})$ is isometric to the warping product $I_\alpha \times_{w} \mathbb{S}^1$ where the interval $I_\alpha = (\log \alpha/\alpha, \infty)$ with the warping function given by $w(v) = e^{-\alpha v}$. If $0<\alpha<1$, then $0\in I_\alpha$. Hence, for every $L>0$, the interval $(0,L)$ is compactly contained in $I_\alpha$, and therefore we may choose the relatively compact domain $\Omega=\pi^{-1}((0,L))\subset \mathcal S$. In this case, $\pi(\Omega)=(0,L)$. Moreover, the fibers are circles of radius $e^{-\alpha t}$, and hence by Example \ref{expwap} item (1) we have $S(t)= {\rm vol}(\pi^{-1}(t))=w(t){\rm vol}(\mathbb{S}^1) = 2\pi e^{-\alpha t}$. The geometric scheme is illustrated in Figure \ref{fighorn}
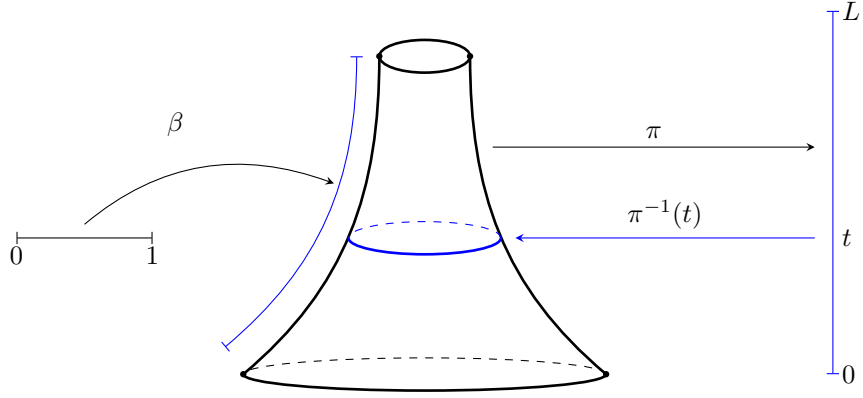
\begin{figure}[htb!]
 \hfill
 \def\r{7}
	\begin{tikzpicture}[scale=0.6,>=stealth]
		
		\def\v{1}
		\def\k{1}
		\draw[line width=1pt] (\k,\k) arc (0:180:1.0cm and 0.36*\v cm);
		\draw[line width=1pt] (\k,\k) arc (360:180:1.0cm and 0.36*\v cm);
		
		\def\u{4}
		\def\w{6}
		\draw[line width=1pt] (\u,-\w) arc (360:180:4.0cm and 0.36*\v cm);
		\draw[dashed] (\u,-\w) arc (0:180:4.0cm and 0.36*\v cm);
		
		\fill (\k,\k) circle (2pt);
		\fill (\u,-\w) circle (2pt);
		
		\fill (-\k,\k) circle (2pt);
		\fill (-\u,-\w) circle (2pt);
		
		\def\a{40}
		\draw[line width=1pt] (\k,\k) to[out=-90,in=180-\a] (\u,-\w);
		\draw[line width=1pt] (-\k,\k) to[out=270,in=\a] (-\u,-\w);

		\draw[line width=1pt,blue] (1.68,-3) arc (360:180:1.68cm and 0.36*\v cm);
		\draw[dashed,blue] (1.68,-3) arc (0:180:1.68cm and 0.36*\v cm);
		
		\draw[|-|,blue] (-1.5*\k,\k) to[out=270,in=\a] (-1.1*\u,-0.9*\w);

		\draw[|-|] (-9,-3) node[below] {$0$} -- (-6,-3) node[below] {$\mathrm{1}$};
		
		\draw[->] (-7.5,-2.7) to[out=40,in=160] (-2,-1.8);
		\node (t1) at (-5.5,-0.5) {$\beta$};
		
		\draw[|-|, blue] (9,2) node[right] {$\textcolor{black}{L}$} -- (9,-6) node[right] {$\textcolor{black}{0}$};
		
		\draw[<-,blue] (2.0,-3) -- (8.6,-3) node[above,pos=0.5] {$\textcolor{black}{\pi^{-1}(t)}$};
		
		\node[right] (t2) at (9,-3) {$t$};
		
		\draw[->] (1.5,-1) -- (8.6,-1) node[above,pos=0.5] {$\pi$};
	\end{tikzpicture}

 \caption{Precompact domain $\Omega \subset \mathcal{S}$ in the pseudosphere, obtained as the inverse image of the interval $(0,L)$ under the Riemannian submersion $\pi$.}
 \label{fighorn}
 \end{figure}

We may therefore apply Theorem \ref{teo:sum} with $S(t)=2\pi e^{-\alpha t}$ to derive the following formula for the harmonic sum of the basic eigenvalues.
\begin{equation*}\label{formspec1}
 \begin{aligned}
 \sum_{k=1}^{\infty}\frac{1}{\lambda_k^{\mathcal{D}}(\Omega)} =& \ \left( \int_{0}^L\frac{dt}{S(t)} \right)^{-1} \int_0^L S(x)\left(\int_0^x\frac{dt}{S(t)}\int_x^L\frac{dt}{S(t)}\right)dx \\
 =& \ \frac{1}{\alpha^2(e^{\alpha L}-1)}\left( e^{\alpha L}(\alpha L -2) + \alpha L + 2 \right).
 \end{aligned}
\end{equation*}
We now substitute $S(t)=2\pi e^{-\alpha t}$ into the weighted Laplacian formula in \eqref{weightedlaplacian_en}, which produces the following differential equation:
\begin{equation*}\label{laplacianwith}
\left\{\begin{aligned}
f''(t) -\alpha f'(t) +\lambda f(t) &= 0 \quad \text{ for } x \in (0,L),\\
f&= 0 \quad \text{ for } x\in \{0\}\cup \{L\}.
\end{aligned}
\right.
\end{equation*}
The solution to this system is obtained via the Liouville transformation by setting $f(t)=e^{(\alpha/2)t}u(t)$. Consequently, we obtain
\begin{equation*}
\left\{\begin{aligned}
u''(t) + \left(\lambda - \frac{\alpha^2}{4} \right) u(t) &= 0 \quad \text{ for } x \in (0,L),\\
u(t) &= 0 \quad \text{ for } t\in \{0\}\cup \{L\}.
\end{aligned}
\right.
\end{equation*}

The function $u$ inherits the Dirichlet boundary conditions from $f$. 
Hence, equation above admits a nontrivial solution if and only if $\lambda-\alpha^2/4>0$. Moreover, a solution of the weighted problem above, up to a multiplicative constant, is given by $f(t)=e^{(\alpha/2)t}\sin((n\pi/L)t)$. 
In this case, the existence of a nontrivial solution satisfying the Dirichlet boundary conditions requires that
\begin{equation*}
 \lambda_k = \frac{\alpha^2}{4} + \frac{k^2 \pi^2 }{L^2} \cdot
\end{equation*}
Therefore, summing over $k$ the reciprocals of the values obtained and comparing with the basic eigenvalue sum obtained previously yields the following formula for the generalized Basel-type series:
\begin{equation*}
\begin{aligned}\label{Basel}
 \sum_{k=1}^{\infty} \dfrac{1}{(\alpha L)^2 + (2 k \pi )^2} =& \frac{1}{4L^2}\sum_{k=1}^{\infty}\frac{1}{\lambda_k} \\
 =& \frac{1}{4L^2} \left(  \sum_{k=1}^{\infty}\frac{1}{\lambda_k^{\mathcal{D}}(\Omega)} \right) \\
 =& \frac{1}{4 (\alpha L)^2(e^{\alpha L}-1)}\left( e^{\alpha L}(\alpha L -2) + \alpha L + 2 \right) \\
 =& \frac{e^{\alpha L}+1}{4 (\alpha L)(e^{\alpha L}-1)}-\frac{1}{2(\alpha L)^2}\\
 =& \frac{1}{4\alpha L}\coth{\left( \frac{\alpha L}{2} \right)} - \frac{1}{2(\alpha L)^2} \cdot
 \end{aligned}
 \end{equation*}
Since $L>0$ can be chosen arbitrarily, the change of variables $z=\alpha L/(2\pi)$ ranges over positive values. Therefore, the identity in \eqref{Basel}, which holds for every $0<\alpha<1$, recovers the Mittag-Leffler expansion of $\pi\coth(\pi z)$ after this change of variables; see \cite[Chapter 7]{Ahlfors1979}.
\end{enumerate}
\end{example}

We now focus on establishing criteria for explicit bounds on the fiber volume function, obtaining a uniform comparison criterion for the sum of the inverses of the basic eigenvalues.

\begin{theorem}\label{thm:comparison} Let $M, N_1$ and $N_2$ be Riemannian manifolds. Suppose that there exist Riemannian submersions $\pi \colon M \to [a,b]$, $\pi_1\colon N_1 \to [a,b]$ and $\pi_2 \colon N_2 \to [a,b]$ with compact fibers of basic mean curvature and fiber volumes given by $S_M(t)=\operatorname{vol}(\pi^{-1}(t))$, $S_{N_1}(t)=\operatorname{vol}(\pi_1^{-1}(t))$, $S_{N_2}(t)=\operatorname{vol}(\pi_2^{-1}(t))$. Assume that there exist constants $\alpha, \beta \in (0,\infty)$ such that $S_{N_1}(t)=\alpha$, $S_{N_2}(t)=\beta$ and $\alpha \leq S(t) \leq \beta$. Then the following statements hold:
\begin{enumerate}
 \item $\displaystyle  \sum_{k=1}^{\infty}\frac{1}{\lambda_k^{\mathcal{D}}(\Omega_{N_1})} =  \sum_{k=1}^{\infty}\frac{1}{\lambda_k^{\mathcal{D}}(\Omega_{N_2})}$;
 \item $\displaystyle\left( \dfrac{\alpha}{\beta} \right)^2 \left(  \sum_{k=1}^{\infty}\frac{1}{\lambda_k^{\mathcal{D}}(\Omega_N)} \right) \leq  \sum_{k=1}^{\infty}\frac{1}{\lambda_k^{\mathcal{D}}(\Omega_M)} \leq \left( \dfrac{\beta}{\alpha} \right)^2 \left(  \sum_{k=1}^{\infty}\frac{1}{\lambda_k^{\mathcal{D}}(\Omega_N)}\right)$.
\end{enumerate}
Here $\Omega_M = \pi^{-1}([a,b])$ and $\Omega_{N_i} = \pi_i^{-1}([a,b])$ for $i=1,2$.
\end{theorem}

\begin{proof} For the proof of item (1) it suffices to note that if $S_{N_i}(t) = \alpha_i$ for some $\alpha_i \in (0,\mathbb{R})$, then by Theorem \ref{teo:sum} we have 
\begin{equation}\label{uniformsum}
 \sum_{k=1}^{\infty}\frac{1}{\lambda_k^{\mathcal{D}}(\Omega_{N_i})} = \frac{(b-a)^2}{6} \cdot
\end{equation}
Therefore, the reciprocal sum of the basic eigenvalues does not depend on the constant $\alpha_i$ chosen initially. For the proof of item (2), since there exist constants $\alpha, \beta \in (0,\infty)$ such that $\alpha \leq S(t) \leq \beta$, together with \eqref{uniformsum}, it follows from the formula given in Theorem \ref{teo:sum} that
\begin{equation*}
 \begin{aligned}
\sum_{k=1}^{\infty}\frac{1}{\lambda_k^{\mathcal{D}}(\Omega_M)} &= \left(\int_{a}^b\frac{dt}{S_M(t)}\right)^{-1}\int_{a}^b S_M(x)\left(\int_{a}^x\frac{dt}{S_M(t)}\int_x^b\frac{dt}{S_M(t)}\right)dx \\
 &\geq\left(\int_{a}^b\frac{dt}{\alpha}\right)^{-1}\int_{a}^b \alpha \left(\int_{a}^x\frac{dt}{\beta}\int_x^b\frac{dt}{\beta}\right)dx \\
 &= \left( \frac{\alpha}{\beta} \right)^2 \dfrac{1}{(b-a)}\int_{a}^b (x-a)(b-x)dx \\
 &= \left( \frac{\alpha}{\beta} \right)^2 \left( \sum_{k=1}^{\infty}\frac{1}{\lambda_k^{\mathcal{D}}(\Omega_N)} \right).
 \end{aligned}
\end{equation*}
The index $i$ in $N_i$ can be omitted due to item (1). The upper bound in item (2) follows by symmetry, exchanging the roles of $\alpha$ and $\beta$.
\end{proof}

The above theorem allows us to state the following corollary.

\begin{corollary}
Let $M$ be a Riemannian manifold. Let $\Omega\subset M$ be a bounded precompact domain with smooth boundary. Suppose that there exists a Riemannian submersion $\pi \colon \Omega \to [a, b]$ with compact fibers of a basic mean curvature vector field. Assume that there exist constants $\alpha, \beta \in (0,\infty)$ such that $\alpha \leq S(t) \leq \beta$. Then:
$$
\lambda_{1}^{\mathcal{D}}(M)\geq \left(\frac{\alpha}{\beta}\right)^2\frac{6}{(b-a)^{2}} \cdot
$$ 
\end{corollary}

\subsection*{Acknowledgements:}
The authors would like to thank Universitat Jaume I (UJI), where this research was carried out. The second author gratefully acknowledges the Universidade Federal do Ceará (UFC) for its institutional support, which made his doctoral research stay at UJI possible. The authors also thank \emph{G. Pacelli Bessa} for valuable conversations and insights related to this work, and \emph{Vicente Miquel} for providing notes on Riemannian submersions. The first author was partially supported by the project AICO/2023/035 funded
by Conselleria d'Educació, Cultura, Universitats i Ocupació. The second author was partially supported by CAPES-Brazil under Grant No. 88881.126989/2025-01.

\def\cprime{$'$} \def\polhk#1{\setbox0=\hbox{#1}{\ooalign{\hidewidth \lower1.5ex\hbox{`}\hidewidth\crcr\unhbox0}}} \def\polhk#1{\setbox0=\hbox{#1}{\ooalign{\hidewidth \lower1.5ex\hbox{`}\hidewidth\crcr\unhbox0}}} \def\polhk#1{\setbox0=\hbox{#1}{\ooalign{\hidewidth \lower1.5ex\hbox{`}\hidewidth\crcr\unhbox0}}} \def\cprime{$'$} \def\cprime{$'$} \def\cprime{$'$} \def\cprime{$'$} \def\cprime{$'$}

\end{document}